\documentclass[12pt, reqno]{amsart}

\usepackage{fullpage}   
\usepackage{amssymb}    
\usepackage{enumitem}
\usepackage{url}

\theoremstyle{plain}
\newtheorem*{main}{Theorem}
\newtheorem{theorem}{Theorem}[section]
\newtheorem{lemma}[theorem]{Lemma}
\newtheorem{proposition}[theorem]{Proposition}
\newtheorem{corollary}[theorem]{Corollary}
\newtheorem{problem}[theorem]{Problem}
\theoremstyle{remark}
\newtheorem*{acknowledgment}{Acknowledgment}

\numberwithin{equation}{section}

\newcommand{\seclabel}[1]{\label{sec:#1}}   
\newcommand{\thmlabel}[1]{\label{thm:#1}}   
\newcommand{\lemlabel}[1]{\label{lem:#1}}   
\newcommand{\corlabel}[1]{\label{cor:#1}}   
\newcommand{\prplabel}[1]{\label{prp:#1}}   
\newcommand{\eqnlabel}[1]{\label{eqn:#1}}   

\newcommand{\secref}[1]{\ref{sec:#1}}   
\newcommand{\thmref}[1]{\ref{thm:#1}}   
\newcommand{\lemref}[1]{\ref{lem:#1}}   
\newcommand{\prpref}[1]{\ref{prp:#1}}   
\newcommand{\eqnref}[1]{\eqref{eqn:#1}} 

\newcommand{\by}[1]{\overset{\eqnref{#1}}=}  
\newcommand{\byx}[1]{\overset{#1}=}          

\newcommand{\ld}{\backslash}                    
\newcommand{\rd}{/}                             

\title{Axioms for Unary Semigroups Via Division Operations}

\author{Jo\~{a}o Ara\'{u}jo}
\author{Michael Kinyon}

\address[Ara\'{u}jo]{Centro de \'{A}lgebra \\
Universidade de Lisboa \\
1649-003 Lisboa \\ Portugal}
\email{\url{jaraujo@ptmat.fc.ul.pt}}

\address[Kinyon]{Department of Mathematics \\
University of Denver \\ 2360 S Gaylord St \\ Denver, Colorado 80208 USA}
\email{\url{mkinyon@math.du.edu}}

\begin{document}

\maketitle

\begin{center}
{\em In memory of Takayuki Tamura (1919-2009)}
\end{center}

\begin{abstract}
When a semigroup has a unary operation, it is possible to define two binary operations, namely, left and right division. In addition it
is well known that groups can be defined in terms of those two divisions. The aim of this paper is to extend those results to other classes of unary semigroups. In the first part of the paper we provide characterizations for several classes of unary semigroups, including (a special class of) E-inversive, regular, completely regular, inverse, Clifford, etc., in terms of left and right division. In the second part we solve a problem that was posed elsewhere. The paper closes with a list of open problems.
\end{abstract}

\section{Introduction}

A unary semigroup $(S,\cdot,{}')$ is a semigroup $(S,\cdot)$ together with an
additional unary operation ${}':S\to S$. Naturally linked to any unary semigroup
are two binary operations, called \emph{left} and \emph{right divisions}, respectively,
and defined by:
\begin{equation}
\eqnlabel{ldrd}
x\ld y = x' y \qquad\text{and}\qquad x\rd y = xy' \,.
\end{equation}
This defines a bimagma (a set with two binary operations) $(S,\ld,\rd)$.
We will refer to $(S,\ld,\rd)$,
with $\ld$ and $\rd$ defined by \eqnref{ldrd}, as the \emph{division bimagma}
of the unary semigroup $(S,\cdot,{}')$. Thus \eqnref{ldrd} defines a functor
$(S,\cdot,{}') \rightsquigarrow (S,\ld,\rd)$ from unary semigroups to their
division bimagmas.

One of the goals of defining a class of algebras in terms of a class of algebras of a different type is that many properties/problems become obvious in the new setting, while difficult to spot in the first. This type of research, very popular some years ago, is now attracting renewed interest as assistance from computational tools allows attacks on problems that not long ago seemed difficult if not impossible. 
Therefore the main theme of this paper is to characterize various classes of unary semigroups
in terms of their division bimagmas. There are many such characterizations for
groups, \emph{e.g.}, \cite{kimura}. Tamura \cite{Tamura} seems to have been the first
to find characterizations for more general classes of unary semigroups, specifically, regular involuted
and inverse semigroups in terms of their division bimagmas. Tamura's work was followed
up in \cite{AM}.

This paper has two main parts. In \S\secref{new}, we offer new characterizations
of several classes of unary semigroups. In \S\secref{Tamura},
we address one of the problems raised in \cite{AM}.

The division bimagmas of all unary semigroups we consider in this paper will
satisfy
\begin{align*}
(x \ld y)\rd z &= x\ld (y\rd z) \tag{B1} \\
x\rd y' &= x'\ld y\,. \tag{B2}
\end{align*}
In fact, (B1) clearly holds in any unary semigroup; it is just a consequence
of associativity and \eqnref{ldrd}. Properties (B1) and (B2) allow us to
reconstruct the unary operation ${}'$ and the semigroup
multiplication $\cdot$ by
\begin{equation}
\eqnlabel{reconstruct}
x' = (x\ld x)\rd x = x\ld (x\rd x)
\qquad \text{and} \qquad
x\cdot y = x\rd y' = x'\ld y\,.
\end{equation}
Here we think of (B1) as implying that $x'$ is well-defined in \eqnref{reconstruct}.
In (B2), we view $x'$ as a shorthand for $(x\ld x)\rd x = x\ld (x\rd x)$. Thus we
do not consider ${}'$ to be part of the signature of a bimagma $(S,\ld,\rd)$.

It is easy to see that starting with a bimagma $(S,\ld,\rd)$ satisfying
(B1), (B2), and defining ${}'$ and $\cdot$ by \eqnref{reconstruct}, we obtain a
unary semigroup $(S,\cdot,{}')$. Indeed,
$(xy)z = (x'\ld y)\rd z' \byx{(B1)} x'\ld (y\rd z') = x(yz)$.
Thus we have a functor $(S,\ld,\rd) \rightsquigarrow (S,\cdot,{}')$ from
bimagmas satisfying (B1) and (B2) to unary semigroups.

We now put some of the main results of the first part of this paper together into the following
summary. More precise statements will be given in the next section, where we will also
recall the definitions of the various classes of unary semigroups.

\begin{main}
\thmlabel{main}
In the following table, the class of unary semigroups $(S,\cdot,{}')$ on the left is
characterized in terms
of their division bimagmas $(S,\ld,\rd)$ by the identities on the right.
\begin{center}
\begin{tabular}{l|l}
$(S,\cdot,{}')$ &   $(S,\ld,\rd)$ \\
\hline\hline
$E$-inversive, \eqnref{tech} &
(B1), $x'\ld y' = x\rd y$, $x'\rd y' = x\ld y$ \\
\hline
regular, $x'' = x$ & (B1), (B2), $x'\ld (x\ld x) = x$ \\
\hline
regular involuted & (B1), (B2), $x \rd (x \ld x)  = x$,
$(x\rd y)' = y\rd x$ \\
\hline
regular involuted, $(x(xx)'x)' = x(xx)'x$ &
(B1), (B2), $x \rd (x \ld x)  = x$,
$(x\rd y)' = y\rd x$,\\
& \qquad $(x\rd (y\ld x))\rd (y\ld y) = ((x\rd (y\ld x))\rd y)\ld y$ \\
\hline
inverse & (B1), (B2), $x\rd (x\ld x) = x$, \\
& \qquad $(x\ld x)\rd (y\rd y) = (y\rd y)\rd (x\ld x)$ \\
\hline
completely regular & (B1), (B2), $x\rd (x\ld x) = x$,
$x\rd x = x\ld x$\\
\hline
Clifford & (B1), (B2), $x\rd (x\ld x) = x$, $x\rd x = x\ld x$ \\
& \qquad $(x\ld x)\rd (y\rd y) = (y\rd y)\rd (x\ld x)$\,. \\
\hline
\end{tabular}
\end{center}
\noindent In each case, the functors
$(S,\cdot,{}') \rightsquigarrow (S,\ld,\rd)$ given by \eqnref{ldrd} and
$(S,\ld,\rd) \rightsquigarrow (S,\cdot,{}')$ given by
\eqnref{reconstruct} are mutually inverse.

Further, in each case, the identities on the right side of the table are independent.
\end{main}

The second part of this paper follows more directly the work of Tamura.
Tamura was motivated by work of Kimura and Sen \cite{kimura}, who showed
that (in our terminology) groups are precisely those unary semigroups whose
division bimagmas satisfy the following two identities:
\begin{equation}
\eqnlabel{kimura-sen}
(x\rd y)\ld z = y\rd (z\ld x) \qquad\text{and}\qquad (x\rd y)\ld x = y\,.
\end{equation}

Tamura \cite{Tamura}
characterized regular involuted unary semigroups and inverse
semigroups in a similar way. The identities he considered are the following:
\begin{align*}
(x \rd y) \ld z &= y \rd (z \ld x) & \text{(T1)} &&
(x \rd x) \ld x &= x               & \text{(T2)} \\
x\rd y' &= x'\ld y & \text{(T3)} &&
(x\rd y)' &= y\rd x & \text{(T4)} \\
(x\rd x)\rd (y\rd y) &= (y\rd y)\rd (x\rd x)\,. & \text{(T5)} && &
\end{align*}
(Note that (T3) is just (B2).)
Tamura's characterizations are as follows (\cite{Tamura}, Theorem 3.1).

\begin{proposition}
\prplabel{tamura1}
\begin{enumerate}[label=\emph{\arabic*)}]
\item Let $(S,\cdot,{}')$ be a regular involuted unary semigroup. Then the
division bimagma $(S,\ld,\rd)$ satisfies (T1)--(T4).
Conversely, let $(S,\ld,\rd)$ be a bimagma satisfying (T1)--(T4).
Then $(S,\cdot,{}')$, with ${}'$ and $\cdot$ defined by \eqnref{reconstruct},
is a regular involuted unary semigroup.
\item Let $(S,\cdot,{}')$ be an inverse semigroup. Then the
division bimagma $(S,\ld,\rd)$ satisfies the identities (T1)--(T5).
Conversely, let $(S,\ld,\rd)$ be a bimagma satisfying (T1)--(T5).
Then $(S,\cdot,{}')$, with ${}'$ and $\cdot$ defined by \eqnref{reconstruct},
is an inverse semigroup.
\end{enumerate}
\end{proposition}

Tamura raised the question of the independence of the axioms
(T1)--(T5). This question was answered by the first-named author and McCune
(\cite{AM}, Theorem 5.2).

\begin{proposition}
\prplabel{AM}
A bimagma $(S,\ld,\rd)$ satisfying (T1)--(T3) satisfies (T4).
\end{proposition}

The proof in \cite{AM} is the output of a proof found by the automated
deduction tool \textsc{Prover9} \cite{Prover9}.
In fact, the statement of the result in \cite{AM} is that
(T1)--(T3) and (T5) imply (T4), but an examination of the proof output
shows that axiom (T5) was never used.

Thus axioms (T1)--(T3) characterize division bimagmas of regular involuted unary semigroups
and axioms (T1)--(T3), (T5) characterize division bimagmas of inverse semigroups.
In addition, axioms (T1)--(T3), (T5) are all independent \cite{AM}.

A problem posed in \cite{AM} is to find a humanly readable proof of
Proposition \prpref{AM}. In \S\secref{Tamura}, we give such a proof.

The paper closes with a section of open problems.

We conclude this introduction with a remark on terminology.
Both \cite{Tamura} and \cite{AM} adapted the classical term ``groupoid''
(a set with a binary operation)
to the present setting, referring to the structures $(S,\ld,\rd)$ as
``bigroupoids''. However, the categorical usage of ``groupoid''
(a small category with inverses) seems to have generally eclipsed
the older usage. Thus we have moved to the Bourbaki term \emph{magma}
for a set with a binary operation, with the obvious adaptation
\emph{bimagma} for our particular situation.

\section{New Characterizations}
\seclabel{new}

We recall the definitions of
the classes of unary semigroups we will consider. They are defined in terms
of various subsets of the following properties:
\begin{align*}
x'xx' &= x' & \text{(I1)} && (x(xx)'x)' &= x(xx)'x & \text{(I5)} \\
xx'x &= x & \text{(I2)} && (xy)' &= y'x' & \text{(I6)} \\
x'' &= x & \text{(I3)} &&  xx' &= x'x & \text{(I7)} \\
(xx')' = xx',&\quad (x'x)'=x'x & \text{(I4)} && xx'y'y &= y'yxx' & \text{(I8)}
\end{align*}
A unary semigroup $(S,\cdot,{}')$ is $E$-\emph{inversive} if it
satisfies (I1). In this case, $(S,\cdot)$ is $E$-inversive in the usual
sense and conversely, each $E$-inversive semigroup has
a choice of weak inverse ${}' : S\to S$ satisfying (I1). Note that
comparing \eqnref{ldrd} with
our reconstruction of ${}'$ in \eqnref{reconstruct} already indicates why
$E$-inversive semigroups are the most basic class we will consider.

$(S,\cdot,{}')$ is \emph{regular} if it satisfies (I1) and (I2).
In this case, $(S,\cdot)$ is regular in the usual sense, and conversely,
each regular semigroup has a choice of inverse ${}' : S\to S$
satisfying (I1), (I2).

Every regular semigroup $(S,\cdot)$ has a choice of inverse ${}'$ that
fixes idempotents, that is, $e' = e$ for every idempotent $e\in S$;
simply redefine ${}'$ on the idempotents if necessary.
It is easy to see that in a regular semigroup this property is equivalent to (I5) being an identity. In fact in a unary regular semigroup (thus the unary operation satisfies (I1) and (I2))  $x(xx)'x$ is idempotent and hence, if ${}'$ fixes idempotents, we have (I5); conversely, for $e^2=e\in S$ we claim that $e'=e$. In fact, (I5)  implies  that $e(ee)'e=(e(ee)'e)'$ and hence $ee'e=(ee'e)'$ thus proving the claim.

The equations that make up (I4) can be considered to be weak versions
of (I5); they are not of much independent interest, but we will use
them to facilitate proofs in other cases.

A unary semigroup $(S,\cdot,{}')$ is \emph{involuted}
if it satisfies (I3) and (I6). (See \cite{NS}.)

Among the most distinguished classes of semigroups are
\emph{inverse} semigroup and \emph{completely regular} semigroups.
Inverse semigroups, to which various books have
been dedicated \cite{lawson,paterson,petrich}, are involuted
regular unary semigroups satisfying (I8).
Completely regular semigroups, which have also been the
subject of a book \cite{PR}, are regular unary semigroups satisfying
(I7). Both
inverse and completely regular semigroups have idempotent-fixing inverses.
Finally, a \emph{Clifford} semigroup is a completely regular, inverse
semigroup.

In order to keep our investigations within reasonable bounds,
we will insist that the functors
$(S,\cdot,{}') \rightsquigarrow (S,\ld,\rd)$ defined by \eqnref{ldrd}
and $(S,\ld,\rd) \rightsquigarrow (S,\cdot,{}')$ defined by \eqnref{reconstruct}
are inverses of each other.  Then it follows that $x' = x\ld (x\rd x) = x'xx'$, that is,
(I1) holds. Thus from now on, we will work in classes of unary semigroups
which are, at the very least, $E$-inversive. In addition, we have
$xy = x'\ld y = x'' y$ and $xy = x\rd y' = xy''$, that is, we assume
\begin{equation}
\eqnlabel{tech}
x'' y = xy = x y''\,.
\end{equation}
(This condition will be subsumed by (I3) in all of our regular classes
of unary semigroups.)

Finally, recall that ${}'$ does not appear in the signature of a bimagma; as said after \eqnref{reconstruct}, $x'$ should be understood as a shorthand for
$(x\ld x)\rd x$ (and also to $x\ld (x\rd x)$, by (B1)),  whenever $x'$ appears in an identity involving $\ld$ and $\rd$. As the reader will realize soon without this shorthand the proofs below  would be much more difficult to follow.
\begin{theorem}
\thmlabel{E-inversive}
Let $(S,\cdot,{}')$ be an $E$-inversive unary semigroup satisfying
\eqnref{tech}. Then the division bimagma $(S,\ld,\rd)$ satisfies
the independent identities \emph{(B1)} and
\begin{align}
x' \ld y' &= x \rd y \eqnlabel{comp1}\\
x' \rd y' &= x \ld y\,.  \eqnlabel{comp2}
\end{align}

Conversely, let $(S,\ld,\rd)$ be a bimagma satisfying \emph{(B1)},
\eqnref{comp1} and \eqnref{comp2}. Then $(S,\cdot,{}')$, with ${}'$ and $\cdot$ defined
 by \eqnref{reconstruct}, is an $E$-inversive unary semigroup satisfying \eqnref{tech}.

The functors
$(S,\cdot,{}') \rightsquigarrow (S,\ld,\rd)$ defined by \eqnref{ldrd}
and $(S,\ld,\rd) \rightsquigarrow (S,\cdot,{}')$ defined by \eqnref{reconstruct}
are inverses
\end{theorem}

\begin{proof}
Suppose $(S,\ld,\rd)$ is a bimagma satisfying \emph{(B1)},
\eqnref{comp1} and \eqnref{comp2}. We claim that the unary semigroup $(S,\cdot,{}')$ whose operations are defined by \eqnref{reconstruct} is an $E$-inversive unary semigroup satisfying
\eqnref{tech}. We must first show that the two possible
definitions of $\cdot$ in \eqnref{reconstruct} coincide. Thus we have to prove that (B2) holds.
Note that by \eqnref{comp1} and \eqnref{comp2}
\begin{equation}
\eqnlabel{e-inv-tmp1}
x'' \ld y'' =x'\rd y'= x \ld y \qquad\text{and}\qquad x'' \rd y''=x'\ld y' = x\rd y\,.
\end{equation}
Next we show
\begin{equation}
\eqnlabel{e-inv-tmp2}
x''' = x'
\end{equation}
as follows:
\begin{alignat*}{4}
x''' &\by{reconstruct} (x'' \ld x'') \rd x''
&&\by{comp1} (x'\rd x')\rd x''
&&\by{reconstruct} [(x\ld (x\rd x))\rd x'] \rd x'' \\
&\byx{(B1)} [x\ld ((x\rd x)\rd x')] \rd x''
&&\by{comp2} [x\ld ((x'\ld x')\rd x')] \rd x''
&&\byx{(B1)} [x\ld (x'\ld (x'\rd x'))] \rd x'' \\
&\by{reconstruct} [x \ld x'']\rd x''
&&\byx{(B1)} x\ld [x''\rd x'']
&&\by{e-inv-tmp1} x\ld [x\rd x] \\
&\by{reconstruct} x'\,. && &&
\end{alignat*}
Thus $x\rd y' \by{e-inv-tmp1} x''\rd y''' \by{e-inv-tmp2} x''''\rd y'''
\by{e-inv-tmp1} x''\rd y' \by{comp2} x'\ld y$, so that (B2) holds as claimed.

Since (B1) and (B2) hold, $(S,\cdot,{}')$ is a semigroup. For \eqnref{tech},
we compute $xy'' = x'\ld y'' \by{comp1} x\rd y' = xy$, and a
similar calculation gives the other identity.
We have $x\ld y = x'\rd y' = x'y$ by \eqnref{comp2} and
similarly, $x\rd y = xy'$. Then $E$-inversivity follows easily:
$x'xx' = x'(xx') = x \ld (x\rd x) = x'$.

Conversely, assume that $(S,\cdot,{}')$ is an $E$-inversive semigroup satisfying
\eqnref{tech} and let $(S,\ld,\rd)$ be the associated division bimagma.
As noted before, (B1) holds for any unary semigroup. Also,
$x'\ld y' = x'' y' = x y' = x\rd y$, which is \eqnref{comp1}, and
\eqnref{comp2} is similarly proved.

The assertion regarding inverse functors is now clear.

Finally, we check independence of (B1), \eqnref{comp1} and \eqnref{comp2}.
On a two-element set $S = \{a,b\}$, define $x\rd y = b$,
$x\ld a = a$ and $x\ld b = b$  for all $x,y\in S$. This is easily
seen to satisfy (B1) and \eqnref{comp1}, but not
\eqnref{comp2}. Reversing the roles of $\ld$ and $\rd$ gives a model
showing the independence of \eqnref{comp1}.
Next, on a three-element
set $S = \{a,b,c\}$, define $x \ld x = x$,
$a \ld b = b\ld a = c$, $a\ld c = c\ld a = b \ld c = c\ld b = a$
and $x \rd y = x\ld y$ for all $x\in S$.
This model can be seen to satisfy \eqnref{comp1} and \eqnref{comp2}
but not (B1).
\end{proof}

We will use the assertion about inverse functors implicitly in what
follows. Once we have established that some class of bimagmas
gives, at the very least, an $E$-inversive semigroup satisfying \eqnref{tech},
then we will usually rewrite the division operations in terms of the
semigroup multiplication and unary map.

Next we consider the case of a regular unary semigroup $(S,\cdot,{}')$.
Our choice of inverse is somewhat restricted
by \eqnref{tech}.

\begin{lemma}
\lemlabel{doubleprime}
Let $(S,\cdot,{}')$ be an $E$-inversive unary semigroup satisfying
\eqnref{tech}. Then \emph{(I3)} holds if and only if
$(S,\cdot,{}')$ is a regular unary semigroup.
\end{lemma}

\begin{proof}
The ``only if'' direction is clear. For the converse,
$x = xx'x \byx{(I3)} x''x'x'' \by{tech} x''$.
\end{proof}

\begin{theorem}
\thmlabel{regular}
Let $(S,\cdot,{}')$ be a regular unary semigroup satisfying \emph{(I3)}.
Then the division bimagma $(S,\ld,\rd)$ satisfies the independent
identities \emph{(B1)}, \emph{(B2)} and
\begin{equation}
x' \ld (x \ld x) = x\,. \eqnlabel{reg2}
\end{equation}

Conversely, let $(S,\ld,\rd)$ be a bimagma satisfying (B1), (B2)
and \eqnref{reg2}. Then  $(S,\cdot,{}')$, with ${}'$ and $\cdot$ defined by \eqnref{reconstruct},
is a regular unary semigroup satisfying 
(I3).
\end{theorem}

Note that \eqnref{reg2} can be replaced by its ``mirror image''
$(x\rd x)\rd x' = x$.

\begin{proof}
Suppose $(S,\ld,\rd)$ is a bimagma satisfying (B1), (B2) and \eqnref{reg2}.
By a sequence of calculations, we will show that (I3) holds. Firstly, we compute
\begin{equation}
\eqnlabel{reg-18}
x' \ld (x''\ld x) \byx{(B2)} x'\ld (x'\rd x') \by{reconstruct} x''\,.
\end{equation}
Next, we have
\begin{equation}
\eqnlabel{reg-30}
x'\ld x' \by{reconstruct} x'\ld ((x\ld x)\rd x) \byx{(B1)} (x'\ld (x\ld x))\rd x
\by{reg2} x\rd x\,.
\end{equation}
Third, we compute
\begin{equation}
\eqnlabel{reg-34}
\begin{alignedat}{2}
x''\ld y &\byx{(B2)} x'\rd y' &&\by{reconstruct} (x\ld (x\rd x))\rd y' \\
&\byx{(B1)} x\ld ((x\rd x)\rd y')
&&\by{reg-30} x\ld ((x'\ld x')\rd y') \\
&\byx{(B1)} x\ld (x'\ld (x'\rd y')) &&\byx{(B2)} x\ld (x'\ld (x''\ld y))\,.
\end{alignedat}
\end{equation}
Now in \eqnref{reg-34}, set $y = x$ to get
\begin{equation}
\eqnlabel{reg-45}
x''\ld x = x\ld (x'\ld (x''\ld x)) \by{reg-18} x\ld x''\,.
\end{equation}
Next, we compute
\begin{equation}
\eqnlabel{reg-46a}
x''\ld x'' \by{reg-30} x'\rd x' \byx{(B2)} x''\ld x
\by{reg-45} x\ld x''\,.
\end{equation}
Next, we have
\begin{equation}
\eqnlabel{reg-39a}
\begin{alignedat}{2}
x'\ld (x\ld (x'\ld y)) &\byx{(B2)} x'\ld (x\ld (x\rd y'))
&&\byx{(B1)} x'\ld ((x\ld x)\rd y') \\
&\byx{(B1)} (x'\ld (x\ld x))\rd y')
&&\by{reg2} x\rd y' \\
&\byx{(B2)} x'\ld y\,.
\end{alignedat}
\end{equation}
Now in \eqnref{reg-39a}, take $y = x''\ld x$ and use \eqnref{reg-18} as
follows:
\begin{equation}
\eqnlabel{reg-39}
x'\ld (x\ld x'') \by{reg-18}
x'\ld (x\ld (x'\ld (x''\ld x))) \by{reg-39a}
x'\ld (x''\ld x) \by{reg-18} x''\,.
\end{equation}
Next we show
\begin{equation}
\eqnlabel{reg-49}
\begin{alignedat}{2}
x'''\ld x''' &\by{reg-30} x''\rd x'' &&\by{reconstruct} (x'\ld (x'\rd x')) \rd x'' \\
&\byx{(B1)} x'\ld ((x'\rd x')\rd x'') &&\byx{(B2)} x'\ld ((x''\ld x)\rd x'') \\
&\byx{(B1)} x'\ld (x''\ld (x\rd x'')) &&\byx{(B2)} x'\ld (x''\ld (x'\ld x')) \\
&\by{reg2} x'\ld x'\,. &&
\end{alignedat}
\end{equation}
Now we have
\begin{equation}
\eqnlabel{reg-54}
x''' \by{reg-39} x''\ld (x'\ld x''')
\by{reg-46a} x''\ld (x'''\ld x''')
\by{reg-49} x''\ld (x'\ld x') \by{reg2} x'\,.
\end{equation}
Next, we have
\begin{equation}
\eqnlabel{reg-55}
\begin{alignedat}{2}
x\ld (x'\ld y) &\by{reg-54}
x\ld (x'''\ld y) &&\by{reg-34}
x\ld (x'\ld (x''\ld (x'''\ld y))) \\
&\by{reg-34} x''\ld (x'''\ld y)
&&\by{reg-54} x''\ld (x'\ld y)\,.
\end{alignedat}
\end{equation}
Take $y = x\ld x$ in \eqnref{reg-55} to get
\begin{equation}
\eqnlabel{reg-61a}
x\ld x \by{reg2} x\ld (x'\ld (x\ld x))
\by{reg-55} x''\ld (x'\ld (x\ld x))
\by{reg2} x''\ld x\,,
\end{equation}
and so
\begin{equation}
\eqnlabel{reg-61}
x\ld x'' \by{reg-45} x''\ld x \by{reg-61a} x\ld x\,.
\end{equation}
Finally,
\[
x'' \by{reg-39} x'\ld (x\ld x'') \by{reg-61} x'\ld (x\ld x)
\by{reg2} x\,.
\]
Thus (I3) holds.

Now we may compute
\[
x'\ld y' \byx{(B2)} x\rd y'' \byx{(I3)} x\rd y
\quad\text{and}\quad
x'\rd y' \byx{(B2)} x''\ld y \byx{(I3)} x\ld y\,.
\]
Thus \eqnref{comp1} and \eqnref{comp2} hold.
By Theorem \thmref{E-inversive}, $(S,\cdot,{}')$ is an $E$-inversive unary
semigroup. Since (I3) holds, $(S,\cdot,{}')$ is a regular
unary semigroup by Lemma \lemref{doubleprime}.

Suppose now $(S,\cdot,{}')$ is a regular unary semigroup and that (I3)
holds. Then $(S,\cdot,{}')$ is $E$-inversive and satisfies \eqnref{tech},
so that (B1) and (B2) hold by Theorem \thmref{E-inversive}.
For \eqnref{reg2}, we compute
$x' \ld (x \ld x) = (x'(xx'))'(x'x)
\byx{(I1)} x''x'x \byx{(I3)} xx'x \byx{(I2)} x$.

Finally, we check independence of the axioms.
On a two-element set $S = \{a,b\}$, define $x \rd x = x$, $a\rd b = a$,
$b\rd a = b$ and $x\ld y = y\rd x$ for all $x,y\in S$. Then $(S,\ld,\rd)$ is a model
satisfying (B1), \eqnref{reg2} but not (B2).
Next, on $S = \{a,b\}$, define $x\ld y = x\rd y = b$ for all $x,y\in S$.
This model satisfies (B1), (B2) but not \eqnref{reg2}.
Finally, on $S = \{a,b\}$, define $x\rd x = b$, $a\rd b = b\rd a = a$,
$x\ld a = a$ and $x\ld b = b$ for all $x\in S$. This model satisfies \eqnref{reg2},
(B2) but not (B1).
\end{proof}

The next class of regular unary semigroups we consider is probably not
of much independent interest. In particular, we do not see that
(I4) reveals any structural features about the underlying
regular semigroup as, for instance, (I5) does, because in general, not
every idempotent in a regular unary semigroup satisfying (I4) will have
the form $xx'$ or $x'x$.  A minimal example (unique up to isomorphism) is
\[
\begin{tabular}{|r|rrrr|}
\hline
* & 0 & 1 & 2 & 3\\
\hline
    0 & 0 & 3 & 0 & 3 \\
    1 & 2 & 1 & 2 & 1 \\
    2 & 2 & 1 & 2 & 1 \\
    3 & 0 & 3 & 0 & 3 \\
\hline
\end{tabular}
\]
with $0'=1$ and $x'=0$ elsewhere; in this band $00'=01=3\neq 0$ and $0'0=10=2\neq 0$.

In any case, we include the characterization of this class here because it facilitates
the proof in the case of a regular unary semigroup with an idempotent
fixing inverse. First, for convenience, we introduce a key identity:
\[
x\rd (x\ld x) = x\,. \tag{B3}
\]

\begin{theorem}
\thmlabel{strange}
Let $(S,\cdot,{}')$ be a regular unary semigroup satisfying \emph{(I3)} and \emph{(I4)}.
Then the division bimagma $(S,\ld,\rd)$ satisfies the independent identities
\emph{(B1)}, \emph{(B2)}, \emph{(B3)} and
\begin{equation}
\eqnlabel{str3}
(x\rd x)\ld x = x\,.
\end{equation}

Conversely, let $(S,\ld,\rd)$ be a bimagma satisfying \emph{(B1)}, \emph{(B2)},
\emph{(B3)} and \eqnref{str3}. Then  $(S,\cdot,{}')$,
with ${}'$ and $\cdot$ defined by \eqnref{reconstruct}, is a regular unary semigroup
satisfying (I3) and (I4).
\end{theorem}

Note that \eqnref{str3} is just (T2) from the Introduction.

\begin{proof}
Suppose $(S,\ld,\rd)$ satisfies (B1), (B2),
(B3) and \eqnref{str3}. We will verify \eqnref{reg2}. Firstly,
\begin{equation}
\eqnlabel{str-goal2}
\begin{alignedat}{2}
(x\ld x)' &\by{reconstruct} (x\ld x)\ld [(x\ld x)\rd (x\ld x)]
&&\byx{(B1)} (x\ld x)\ld [x\ld (x\rd (x\ld x))] \\
&\byx{(B3)} (x\ld x)\ld [x\ld x]
&&\byx{(B3)} [x\ld (x\rd (x\ld x))]\ld [x\ld x] \\
&\byx{(B1)} [(x\ld x)\rd (x\ld x)]\ld [x\ld x]
&&\by{str3} x\ld x\,.
\end{alignedat}
\end{equation}
A dual calculation gives
\begin{equation}
\eqnlabel{str-goal1}
(x\rd x)' = x\rd x\,.
\end{equation}
Thus $x' \ld (x\ld x) \byx{(B2)} x\rd (x\ld x)' \by{str-goal2} x\rd (x\ld x)
\byx{(B3)} x$. This is \eqnref{reg2}. By Theorem \thmref{regular},
$(S,\cdot,{}')$ is a regular unary semigroup satisfying (I3). By
\eqnref{str-goal2}, we have $(x'x)' = (x\ld x)' = x \ld x = x'x$,
and \eqnref{str-goal1} similarly gives $(xx')' = xx'$. Thus (I4)
holds.

Conversely, suppose $(S,\cdot,{}')$ is a regular unary semigroup satisfying (I3) and
(I4). By Theorem \thmref{regular}, (B1) and (B2) hold.
For (B3), $x\rd (x\ld x) = x(x'x)'
\byx{(I4)} xx'x \byx{(I2)} x$, and \eqnref{str3} is proved similarly.

On $S = \{a,b\}$, define $a\rd x = b\rd b = a$,
$b\rd a = b$ and $x\ld y = y\rd x$ for all $x,y\in S$. This gives a
model satisfying (B1), (B3) and \eqnref{str3}, but not (B2).
Again on set $S = \{a,b\}$, define
$a\rd x = a$, $b\rd x = b$, $a\ld x = b$, $b\ld x = a$ for all $x\in S$.
This is a model satisfying (B1), (B2) and (B3),
but not \eqnref{str3}. Exchanging the roles of $\ld$ and $\rd$ gives a model
satisfying (B1), \eqnref{str3} and (B2),
but not (B3). Finally, on $S = \{a,b,c\}$,
define $a\rd a = b\rd a = a$, $a\rd b = a\rd c = b$, $b\rd b = c\rd c = c$,
$b\rd c = c\rd a = b$ and $x\ld y = y\rd x$ for all $x,y\in S$. This is
a model satisfying (B2), (B3) and \eqnref{str3}, but not (B1).
\end{proof}

Next we characterize regular unary semigroups with an idempotent fixing
inverse.

\begin{theorem}
\thmlabel{idem-respect}
Let $(S,\cdot,{}')$ be a regular unary semigroup satisfying \emph{(I3)} and
\emph{(I5)}. Then the division bimagma $(S,\ld,\rd)$ satisfies
\emph{(B1)}, \emph{(B2)}, \emph{(B3)} and
\begin{equation}
\eqnlabel{ir4}
(x\rd (y\ld x))\rd (y\ld y) = ((x\rd (y\ld x))\rd y)\ld y\,.
\end{equation}

Conversely, let $(S,\ld,\rd)$ be a bimagma satisfying
\emph{(B1)}, \emph{(B2)}, \emph{(B3)} and \eqnref{ir4}.
Then  $(S,\cdot,{}')$, with ${}'$ and $\cdot$ defined by \eqnref{reconstruct} is
a regular unary semigroup satisfying \emph{(I3)} and \emph{(I5)}.
\end{theorem}

We will defer discussing the independence of the identities (B1), (B2), (B3) and \eqnref{ir4}
until Proposition \prpref{indep}.

\begin{proof}
Suppose $(S,\ld,\rd)$ satisfies (B1), (B2), (B3) and \eqnref{ir4}.
Setting $y = x$ in \eqnref{ir4} and using (B3) twice on the
left side and once on the right side, we get $(x\rd x)\ld x = x$,
that is, \eqnref{str3} holds. Thus all conditions
of Theorem \thmref{strange} hold, and so it follows that
$(S,\cdot,{}')$ is a regular unary semigroup satisfying (I3) and
(I4). What remains is to show that (I5) holds. We are going  to show that ${}'$ fixes idempotents.
Thus let $e\in S$ be an idempotent. We translate \eqnref{ir4} into the
semigroup language as
$x(y'x)'(y'y)' = (x(y'x)'y')'y$. Apply (I4) to the left side
and then replace $y$ with $y'$ and use (I3) to get
\begin{equation}
\eqnlabel{ir-tmp2}
x(yx)'yy' = (x(yx)'y)'y'\,.
\end{equation}
Setting $x = y = e$ in \eqnref{ir-tmp2} and using $e^2 = e$, we get
$ee'ee' = (ee'e)'e'$. Apply (I2) to both sides to get
\begin{equation}
\eqnlabel{ir-tmp3}
ee' = e'e'\,.
\end{equation}
Now take $x = y = e'$ in \eqnref{ir-tmp2}. On the left side, we get
$e'(e'e')'e'e'' \by{ir-tmp3} e'(ee')'e'e'' \byx{(I4)} e'ee'e'e''
\by{ir-tmp3} e'eee'e'' = e'ee'e'' \byx{(I3)} e'ee'e \byx{(I1)} e'e$.
On the right side, we have
$(e'(e'e')'e')'e'' \by{ir-tmp3} (e'(ee')'e')'e'' \byx{(I4)} (e'ee'e')'e''
\byx{(I1)} (e'e')'e'' \by{ir-tmp3} (ee')'e'' \byx{(I4)} ee'e''
\byx{(I3)} ee'e \byx{(I2)} e$. Thus $e'e = e$, and so
$e' = (e'e)' \byx{(I4)} e'e = e$, as claimed.

Now suppose $(S,\cdot,{}')$ is a regular unary semigroup satisfying (I3) and (I5).
Since (I4) necessarily holds, Theorem \thmref{strange} implies that $(S,\ld,\rd)$
satisfies (B1), (B2) and (B3). For \eqnref{ir4}, we compute
\[
(x\rd (y\ld x))\rd (y\ld y) = x(y'x)'(y'y)' = x(y'x)'y'y
= (x(y'x)'y')'y = ((x\rd (y\ld x))\rd y)\ld y\,,
\]
where the third equality follows since $x(y'x)'y'$ is an idempotent.
\end{proof}

Next we consider regular involuted semigroups.

\begin{theorem}
\thmlabel{reginvoluted}
Let $(S,\cdot,{}')$ be a regular involuted unary semigroup.
Then the division bimagma $(S,\ld,\rd)$ satisfies \emph{(B1)}, \emph{(B2)},
\emph{(B3)} and
\begin{equation}
\eqnlabel{reginv1}
(x\rd y)' = y\rd x\,.
\end{equation}

Conversely, let $(S,\ld,\rd)$ be a bimagma satisfying \emph{(B1)}, \emph{(B2)},
\emph{(B3)} and \eqnref{reginv1}.
Then $(S,\cdot,{}')$, with ${}'$ and $\cdot$ defined by \eqnref{reconstruct},
is a regular involuted unary semigroup.
\end{theorem}

We will defer discussing the independence of the identities (B1), (B2), (B3) and \eqnref{reginv1}
until Proposition \prpref{indep}.

\begin{proof}
Suppose $(S,\ld,\rd)$ satisfies (B1), (B2), (B3) and \eqnref{reginv1}.
Then
$x \byx{(B3)} x\rd (x\ld x) \by{reginv1} [(x\ld x)\rd x]' \by{reconstruct} x''$.
Thus (I3) holds. By Theorem \thmref{regular}, $(S,\cdot,{}')$ is a regular
unary semigroup. Finally $(xy)' \byx{(I3)} (xy'')' \by{ldrd} (x\rd y')'
\by{reginv1} y'\rd x \by{ldrd} y'x'$, so that (I6) holds.

Conversely, if $(S,\cdot,{}')$ is regular and involuted, then (B1), (B2) and (B3) follow
from Theorem \thmref{strange}, while \eqnref{reginv1} is just
$(x\rd y)' \by{ldrd} (xy')' \byx{(I6)} y''x' \byx{(I3)} yx' \by{ldrd} y\rd x$.
\end{proof}

Putting together the identities of the last two results, we have
the following.

\begin{proposition}
\prplabel{indep}
The identities \emph{(B1)}, \emph{(B2)}, \emph{(B3)}, \eqnref{ir4} and
\eqnref{reginv1} are independent.
\end{proposition}

\begin{proof}
On $S = \{a,b\}$, set $a\rd x = a$, $b\rd x = b$
and $x\ld y = x\rd y$ for all $x,y\in S$. This is a model satisfying
(B1), (B2), (B3) and \eqnref{ir4}, but not \eqnref{reginv1}.

On $S = \{a,b\}$, set $x\rd y = b$, $a\ld a = a$,
$a\ld b = b$ and $b\ld x = b$ for all $x,y\in S$. This is a model
satisfying (B1), (B2), \eqnref{ir4}, \eqnref{reginv1} but not (B3).

The pair of tables on the left below give
a model satisfying, respectively, (B2), (B3), \eqnref{ir4}, \eqnref{reginv1} but not (B1),
The tables on the right give
a model satisfying (B1), (B3), \eqnref{ir4}, \eqnref{reginv1} but not (B2).

\begin{table}[htb]  \centering
\begin{tabular}{r|rrr}
$\rd$ & 0 & 1 & 2\\
\hline
    0 & 0 & 2 & 0 \\
    1 & 2 & 1 & 2 \\
    2 & 0 & 2 & 2
\end{tabular}
\quad
\begin{tabular}{r|rrr}
$\ld$ & 0 & 1 & 2\\
\hline
    0 & 0 & 2 & 0 \\
    1 & 2 & 1 & 2 \\
    2 & 0 & 2 & 2
\end{tabular}
\qquad\qquad
\begin{tabular}{r|rrr}
$\rd$ & 0 & 1 & 2\\
\hline
    0 & 0 & 2 & 2 \\
    1 & 2 & 1 & 2 \\
    2 & 2 & 2 & 2
\end{tabular}
\quad
\begin{tabular}{r|rrr}
$\ld$ & 0 & 1 & 2\\
\hline
    0 & 0 & 1 & 2 \\
    1 & 0 & 1 & 2 \\
    2 & 2 & 2 & 2
\end{tabular}
\end{table}

Finally, the following tables give a model satisfying
(B1), (B2), (B3) and \eqnref{reginv1}, but not \eqnref{ir4}.

\begin{table}[htb]
\centering
\begin{tabular}{r|rrrr}
$\rd$ & 0 & 1 & 2 & 3\\
\hline
    0 & 0 & 2 & 0 & 2 \\
    1 & 3 & 1 & 3 & 1 \\
    2 & 0 & 2 & 0 & 2 \\
    3 & 3 & 1 & 3 & 1
\end{tabular} \hspace{.5cm}
\begin{tabular}{r|rrrr}
$\ld$ & 0 & 1 & 2 & 3\\
\hline
    0 & 0 & 2 & 2 & 0 \\
    1 & 3 & 1 & 1 & 3 \\
    2 & 3 & 1 & 1 & 3 \\
    3 & 0 & 2 & 2 & 0
\end{tabular}
\end{table}
\end{proof}

\begin{corollary}
\corlabel{mixed}
Let $(S,\cdot,{}')$ be a regular involuted unary semigroup satisfying (I5).
Then the division bimagma $(S,\ld,\rd)$ satisfies the independent identities
\emph{(B1)}, \emph{(B2)},
\emph{(B3)}, \eqnref{ir4} and \eqnref{reginv1}.

Conversely, let $(S,\ld,\rd)$ be a bimagma satisfying \emph{(B1)}, \emph{(B2)},
\emph{(B3)}, \eqnref{ir4} and \eqnref{reginv1}.
Then $(S,\cdot,{}')$, with ${}'$ and $\cdot$ defined by \eqnref{reconstruct},
is a regular involuted unary semigroup satisfying (I5).
\end{corollary}

Next we consider inverse semigroups providing a basis different from Tamura's.

\begin{theorem}
\thmlabel{inverse3P}
Let $(S,\cdot,{}')$ be an inverse semigroup.
Then the division bimagma $(S,\ld,\rd)$ satisfies the identities
\emph{(B1)}, \emph{(B2)}, \emph{(B3)} and
\begin{equation}
\eqnlabel{invcase}
(x \ld x) \rd (y\rd y) = (y\rd y) \rd (x\ld x)\,.
\end{equation}

Conversely, let $(S,\ld,\rd)$ be a bimagma satisfying \emph{(B1)}, \emph{(B2)}, \emph{(B3)}
and \eqnref{invcase}. Then $(S,\cdot,{}')$, with ${}'$ and $\cdot$ defined by \eqnref{reconstruct},
is an inverse semigroup.
\end{theorem}

We defer discussing the independence of (B1), (B2), (B3) and \eqnref{invcase} until
Proposition \prpref{cliffindep}.

\begin{proof}
Now suppose $(S,\ld,\rd)$ satisfies (B1), (B2), (B3) and \eqnref{invcase}.
Firstly, we compute
\begin{equation}
\eqnlabel{lem12}
(x\ld x)\rd (x\ld x) \byx{(B1)} x\ld [x\rd (x\ld x)] \byx{(B3)} x\ld x\,.
\end{equation}

Next, we show
\begin{align}
(x\rd x)' &= x\rd x \eqnlabel{lem21} \\
\intertext{and}
(x\ld x)' &= x\ld x\,. \eqnlabel{lem26}
\end{align}
Indeed, we have
\[
(x\rd x)' \by{reconstruct}
[(x\rd x)\ld (x\rd x)]\rd (x\rd x) \by{invcase}
(x\rd x)\rd  [(x\rd x)\ld (x\rd x)] \byx{(B3)} x\rd x\,,
\]
and
\[
(x\ld x)' \byx{(B3)} [x\ld (x\rd (x\ld x))]'
\byx{(B1)} [(x\ld x)\rd (x\ld x)] \by{lem12} x\ld x\,.
\]

Next, we have
\[
x'' \eqnref{reconstruct} (x'\ld x')\rd x' \byx{(B2)} (x'\ld x')'\ld x
\by{lem26} (x'\ld x')\ld x \byx{(B2)} (x\rd x'')\ld x\,,
\]
that is,
\begin{equation}
\eqnlabel{lem31}
x'' = (x\rd x'')\ld x\,.
\end{equation}
This gives us
\[
x'' \by{lem31} (x\rd x'')\ld x \byx{(B3)} (x\rd x'')\ld (x\rd (x\ld x))
\byx{(B1)} [(x\rd x'')\ld x]\rd (x\ld x) \by{lem31} x''\rd (x\ld x)\,,
\]
that is,
\begin{equation}
\eqnlabel{lem32}
x'' = x''\rd (x\ld x)\,.
\end{equation}

For the next step, we show
\begin{equation}
\eqnlabel{lem35}
x''\rd x'' = x'\ld x'\,.
\end{equation}
Indeed, we have
\begin{alignat*}{2}
x''\rd x'' &\byx{(B3)} x''\rd (x'\rd (x'\ld x'))' &&\byx{(B2)} x''\rd (x'\rd (x\rd x''))' \\
&\byx{(B2)} x'''\ld (x'\rd (x\rd x'')) &&\byx{(B1)} (x'''\ld x')\rd (x\rd x'') \\
&\byx{(B2)} (x''\rd x'')\rd (x\rd x'') &&\by{lem31} ([(x\rd x'')\ld x]\rd x'')\rd (x\rd x'') \\
&\byx{(B1)} [(x\rd x'')\ld (x\rd x'')]\rd (x\rd x'') &&\by{reconstruct} (x\rd x'')' \\
&\byx{(B2)} (x'\ld x')' &&\by{lem26} x'\ld x'\,.
\end{alignat*}

Now
\[
x''' \eqnref{reconstruct} x''\ld (x''\rd x'') \eqnref{lem35} x''\ld (x'\ld x')
\byx{(B2)} x'\rd (x'\ld x')' \by{lem21} x'\rd (x'\ld x') \byx{(B3)} x'\,,
\]
that is,
\begin{equation}
\eqnlabel{lem36}
x''' = x''\,.
\end{equation}

And now we can verify (I3) as follows:
\begin{alignat*}{4}
x'' &\by{lem32} x''\rd (x\ld x) &&\by{lem26} x''\rd (x\ld x)'
&&\byx{(B2)} x'''\ld (x\ld x) &&\by{lem36} x'\ld (x\ld x) \\
&\byx{(B2)} x\rd (x\ld x)' &&\by{lem26} x\rd (x\ld x)
&&\byx{(B3)} x\,. &&
\end{alignat*}
It follows from Theorem \thmref{regular} that $(S,\cdot,{}')$ is a regular
unary semigroup satisfying (I3). Thus we have
$(x'x)' = (x\ld x)' \by{lem26} x\ld x = x'x$ and
$(xx')' = (x\rd x)' \by{lem21} x\rd x = xx'$, so that (I4) holds.

Finally, we compute
\begin{align*}
(xx')(y'y) &\byx{(I4)} (xx')(y'y)' \by{reconstruct} (x\rd x)\rd (y\ld y)
\by{invcase} (y\ld y)\rd (x\rd x) \\
&\by{reconstruct} (y'y)'(xx') \byx{(I4)} (y'y)(xx')\,,
\end{align*}
and so (I8) holds. Finally, it is well-known that
(I2), (I3) and (I8) are sufficient to imply that a unary semigroup is an
inverse semigroup in which the unary operation is the natural inverse; the
identity (I6) is, in fact, dependent (see \emph{e.g.}, \cite{AM}).

Conversely, if $(S,\cdot,{}')$ is an inverse semigroup, then (B1), (B2) and (B3) follow from
Theorem \thmref{reginvoluted}. For \eqnref{invcase}, we have
$(x\ld x)\rd (y\rd y) \by{ldrd} (x'x)(yy')' \byx{(I4)} (x'x)(yy') \byx{(I8)}
(yy')(x'x) \byx{(I4)} (yy')(x'x)' \by{ldrd} (y\rd y)\rd (x\ld x)$.
\end{proof}

Finally, we turn to completely regular semigroups.

\begin{theorem}
\thmlabel{involuted}
Let $(S,\cdot,{}')$ be a completely regular semigroup.
Then the division bimagma $(S,\ld,\rd)$ satisfies the identities
\emph{(B1)}, \emph{(B2)}, \emph{(B3)} and
\begin{equation}
\eqnlabel{cr4}
x\rd x = x\ld x\,.
\end{equation}

Conversely, let $(S,\ld,\rd)$ be a bimagma satisfying \emph{(B1)}, \emph{(B2)}, \emph{(B3)}
and \eqnref{cr4}. Then $(S,\cdot,{}')$, with ${}'$ and $\cdot$ defined by \eqnref{reconstruct},
is a completely regular semigroup.
\end{theorem}

We defer discussing the independence of (B1), (B2), (B3) and \eqnref{cr4} until
Proposition \prpref{cliffindep}.

\begin{proof}
Suppose $(S,\ld,\rd)$ satisfies (B1), (B2), (B3) and \eqnref{cr4}.
Firstly, we have
\begin{equation}
\eqnlabel{cr-tmp1}
\begin{alignedat}{2}
(x\ld x)' &\by{reconstruct} (x\ld x)\ld ((x\ld x)\rd (x\ld x))
&&\byx{(B1)} (x\ld x)\ld (x\ld (x\rd (x\ld x))) \\
&\byx{(B3)} (x\ld x)\ld (x\ld x)
&&\by{cr4} (x\ld x)\rd (x\ld x) \\
&\byx{(B1)} x\ld (x\rd (x\ld x))
&&\byx{(B3)} x\ld x\,.
\end{alignedat}
\end{equation}
Thus
$x'\ld (x\ld x) \byx{(B2)} x\rd (x\ld x)' \by{cr-tmp1} x\rd (x\ld x)
\byx{(B3)} x$ and hence \eqnref{reg2} holds. Therefore the conditions of
Theorem \thmref{regular} are satisfied, and so $(S,\cdot,{}')$
is a regular semigroup satisfying (I3). In addition,
$xx' = x\rd x \by{cr4} x\ld x = x'x$ and so (I7) holds. Therefore
$(S,\cdot,{}')$ is completely regular.

Conversely, suppose $(S,\cdot,{}')$ is a completely regular semigroup. Then
(B1), (B2) and (B3) hold by Theorem \thmref{strange}, while
\eqnref{cr4} is just (I7) rewritten.
\end{proof}

\begin{proposition}
\prplabel{cliffindep}
The identities \emph{(B1)}, \emph{(B2)}, \emph{(B3)}, \eqnref{invcase} and \eqnref{cr4}
are independent.
\end{proposition}

\begin{proof}
On $S = \{a,b\}$, define $a\rd x = a$, $b\rd x = b$
and $x\ld y = x\rd y$ for all $x,y\in S$. This is a model satisfying
(B1), (B2), (B3) and \eqnref{cr4}, but not \eqnref{invcase}.

On $S = \{a,b\}$, define $x\rd y = x\ld y = b$
for all $x,y\in S$. This gives a model satisfying
(B1), (B2), \eqnref{invcase} and \eqnref{cr4}, but not (B3).

On $S = \{a,b\}$, define $a\rd a = b$, $a\rd b = a$,
$b\rd x = b$, $x\ld y = b$ for all $x,y\in S$. This gives a model satisfying
(B1), (B3), \eqnref{invcase} and \eqnref{cr4}, but not (B2).

On $S = \{a,b\}$, define $a\rd a = b$, $a\rd b = a$,
$b\rd x = b$ and $x\ld y = x\rd y$ for all $x,y\in S$. This gives a model satisfying
(B2), (B3), \eqnref{invcase} and \eqnref{cr4}, but not (B1).

Finally, the table below gives a model satisfying (B1), (B2), (B3) and \eqnref{invcase},
but not \eqnref{cr4}.

\begin{table}[htb]  \centering
\begin{tabular}{r|rrrrr}
$\rd$ & 0 & 1 & 2 & 3 & 4\\
\hline
    0 & 2 & 4 & 4 & 0 & 4 \\
    1 & 4 & 3 & 1 & 4 & 4 \\
    2 & 4 & 0 & 2 & 4 & 4 \\
    3 & 1 & 4 & 4 & 3 & 4 \\
    4 & 4 & 4 & 4 & 4 & 4
\end{tabular} \qquad
\begin{tabular}{r|rrrrr}
$\ld$ & 0 & 1 & 2 & 3 & 4\\
\hline
    0 & 3 & 4 & 1 & 4 & 4 \\
    1 & 4 & 2 & 4 & 0 & 4 \\
    2 & 0 & 4 & 2 & 4 & 4 \\
    3 & 4 & 1 & 4 & 3 & 4 \\
    4 & 4 & 4 & 4 & 4 & 4
\end{tabular}
\end{table}
\end{proof}

\begin{corollary}
Let $(S,\cdot,{}')$ be a Clifford semigroup.
Then the division bimagma $(S,\ld,\rd)$ satisfies the independent identities
\emph{(B1)}, \emph{(B2)}, \emph{(B3)}, \eqnref{invcase} and \eqnref{cr4}.

Conversely, let $(S,\ld,\rd)$ be a bimagma satisfying \emph{(B1)}, \emph{(B2)}, \emph{(B3)},
\eqnref{invcase} and \eqnref{cr4}. Then $(S,\cdot,{}')$, with ${}'$ and $\cdot$ defined by
\eqnref{reconstruct}, is a Clifford semigroup.
\end{corollary}

\section{Tamura's Problem}
\seclabel{Tamura}

Now we turn the problem that arose from Tamura's work \cite{Tamura}. We
repeat here the relevant identities for the reader's convenience:
\begin{align*}
(x \rd y) \ld z &= y \rd (z \ld x) & \text{(T1)} &&
(x \rd x) \ld x &= x               & \text{(T2)} \\
x\rd y' &= x'\ld y & \text{(T3)} &&
(x\rd y)' &= y\rd x & \text{(T4)}
\end{align*}
There is a tacit assumption here that ${}'$ is well-defined, and we
address this now along with other useful facts.

\begin{lemma}
\lemlabel{joint}
Let $(S,\ld,\rd)$ be a bimagma satisfying \emph{(T1)} and \emph{(T2)}.
Then the identities
\begin{align}
x\rd (x\ld x) &= x \eqnlabel{t2a} \\
x\ld (x\rd x) &= (x\ld x)\rd x \eqnlabel{same}
\end{align}
hold. Define ${}' : S\to S$ by \eqnref{reconstruct}. Then the following
identities also hold:
\begin{align}
(x\ld x)' &= x\ld x \eqnlabel{t4-goal1} \\
(x\rd x)' &= x\rd x\,. \eqnlabel{t4-goal2}
\end{align}
\end{lemma}

(Note that \eqnref{t2a} is just (B3).)

\begin{proof}
For \eqnref{t2a}, we have $x\rd (x\ld x) \byx{(T1)} (x\rd x)\ld x\byx{(T2)} x$.
Thus we obtain \eqnref{same} as follows:
\[
x\ld (x\rd x) \by{t2a} (x\rd (x\ld x))\ld (x\rd x)
\byx{(T1)} (x\ld x)\rd (\underbrace{(x\rd x)\ld x}) \byx{(T2)} (x\ld x)\rd x\,.
\]
For \eqnref{t4-goal1}, we compute
\begin{alignat*}{2}
(x\ld x)' &\by{reconstruct} (x\ld x)\ld (\underbrace{(x\ld x)\rd (x\ld x)})
&&\byx{(T1)} (x\ld x)\ld ((\underbrace{x\rd (x\ld x)})\ld x) \\
&\by{t2a} (x\ld x)\ld (x\ld x)
&&\by{t2a} (\underbrace{(x\rd (x\ld x))\ld x})\ld (x\ld x) \\
&\byx{(T1)} ((x\ld x)\rd (x\ld x))\ld (x\ld x)
&&\byx{(T2)} (x\ld x)\,.
\end{alignat*}
The proof of \eqnref{t4-goal2} is dual to this.
\end{proof}

With this lemma in place, we may now use the shorthand
$x' = x\ld (x\rd x) = (x\ld x)\rd x$ from \eqnref{reconstruct}.

Our first goal in this section is to give a humanly readable proof
of Proposition \prpref{AM}, which we restate here.

\begin{proposition}
\prplabel{main}
Let $(S,\ld,\rd)$ be a bimagma satisfying \emph{(T1)}, \emph{(T2)}
and \emph{(T3)}. Then \emph{(T4)} also holds.
\end{proposition}

We begin with two auxiliary lemmas.

\begin{lemma}
Under the assumptions of Proposition \prpref{main}, (I3) holds.
\end{lemma}
\begin{proof}
We compute
\begin{alignat*}{4}
x'' &\by{reconstruct} x'\ld (x'\rd x')
&&\by{reconstruct} x'\ld (\underbrace{x'\rd (x\ld (x\rd x))})
&&\byx{(T1)} x'\ld ((\underbrace{(x\rd x)\rd x'})\ld x) \\
&\byx{(T3)} x'\ld ((\underbrace{(x\rd x)'}\ld x)\ld x)
&&\by{t4-goal2} x'\ld ((\underbrace{(x\rd x)\ld x})\ld x)
&& \byx{(T2)} x'\ld (x\ld x) \\
&\byx{(T3)} x\rd (x\ld x)'
&&\by{t4-goal1} x\rd (x\ld x)
&&\by{t2a} x\,. \tag*{\qedhere}
\end{alignat*}
\end{proof}

The next lemma provides a number of handy ways of expressing the products
of two or three elements, and how the inversion relates with the two binary operations.

\begin{lemma}
Under the assumptions of Proposition \prpref{main} the following identities hold:
 \begin{align}
x \rd y &= (x\rd y) \rd (y\rd y)  \eqnlabel{step65} \\
x\rd y &= (x\rd x)\ld (x \rd y)  \eqnlabel{step36} \\
x\ld y &= ((y \rd y) \ld x)\ld y  \eqnlabel{step34} \\
x\ld y &= x' \rd y'  \qquad\text{and}\qquad  x' \ld y' = x\rd y  \eqnlabel{t4-goal4} \\
x \ld y'& = x' \rd y     \eqnlabel{t4-goal3} \\
x' \ld (y \rd z) &=x \rd (z \rd y)   \eqnlabel{step38} \\
x\ld (y\rd z)'&= x\ld (z\rd y)   \eqnlabel{step43} \\
(x\rd y)\ld z&=(y\rd x)' \rd z    \eqnlabel{step49}
\end{align}
\end{lemma}

\begin{proof}
We start by proving \eqnref{t4-goal3}.
\[
x'\rd y \byx{(I3)} x'\rd y'' \byx{(T3)} x''\ld y' \byx{(I3)} x\ld y'\,,
\]
and so, in particular, \eqnref{t4-goal4} follows:
\[
x'\rd y' \by{t4-goal3} x\ld y'' \byx{(I3)} x\ld y
\qquad\text{and}\qquad
x'\ld y' \by{t4-goal3} x''\rd y \byx{(I3)} x\rd y\,.
\]

Regarding \eqnref{step36},
\begin{alignat*}{2}
(x\rd x)\ld (x\rd y) &\byx{(T1)} x\rd (\underbrace{(x\rd y)\ld x})
&&\byx{(T1)} x\rd (y\rd \underbrace{(x\ld x)}) \\
&\by{t4-goal1} x\rd (\underbrace{y\rd (x\ld x)'})
&&\byx{(T3)} x\rd (y'\ld (x\ld x)) \\
&\byx{(T1)} (\underbrace{(x\ld x)\rd x})\ld y'
&&\by{reconstruct} x'\ld y' \\
&\by{t4-goal4} x\rd y\,. &&
\end{alignat*}

For \eqnref{step34} we have
\begin{alignat*}{4}
x\ld y &\by{t4-goal4} x'\rd y' &&\by{reconstruct} x'\rd (y\ld (y\rd y))
&&\byx{(T1)} ((y\rd y)\rd x')\ld y \\
&\byx{(T3)} ((y\rd y)'\ld x)\ld y &&\by{t4-goal2} ((y\rd y)\ld x)\ld y\,. &&
\end{alignat*}

Regarding \eqnref{step38},
\begin{alignat*}{4}
x\rd (z\rd y) &\by{t4-goal4} x\rd (z'\ld y')
&&\byx{(T1)} (\underbrace{y'\rd x})\ld z'
&&\by{t4-goal3} (y\ld x')\ld z' \\
&\by{step34} (\underbrace{((x'\rd x')\ld y)\ld x'})\ld z'
&&\by{t4-goal3} (((x'\rd x')\ld y)'\rd x)\ld z'
&&\byx{(T1)} x\rd (\underbrace{z'\ld ((x'\rd x')\ld y)'}) \\
&\by{t4-goal4} x\rd (z\rd ((\underbrace{x'\rd x'})\ld y))
&&\by{t4-goal4} x\rd (\underbrace{z\rd ((x\ld x)\ld y)})
&&\byx{(T1)} x\rd ((y\rd z)\ld (x\ld x)) \\
&\byx{(T1)} ((x\ld x)\rd x)\ld (y\rd z)
&&\by{reconstruct} x'\ld (y\rd z)\,. &&
\end{alignat*}
From this, we get \eqnref{step43}:
\[
x\ld (y\rd z)' \by{t4-goal3} x'\rd (y\rd z) \by{step38} x''\ld (z\rd y)
\byx{(I3)} x\ld (z\rd y)\,.
\]

For \eqnref{step49},
\[
(x\rd y)'\rd z \by{step34} (\underbrace{(z\rd z)\ld (x\rd y)'})\ld z
\by{step43} ((z\rd z)\ld (y\rd x))\ld z
\by{step34} (y\rd x)\ld z\,.
\]

Finally we prove \eqnref{step65}:
\begin{alignat*}{2}
(x\rd y)\rd (y\rd y) &\by{step38} (x\rd y)'\ld (y\ld y)
&&\by{step49} (y\rd x)\ld (y\rd y) \\
&\byx{(T1)} x\rd ((y\rd y)\ld y)
&&\byx{(T2)} x\rd y\,.  \tag*{\qedhere}
\end{alignat*}
\end{proof}

We have everything we need to prove Proposition \prpref{main}.
\begin{proof} (Proposition \prpref{main})

We claim that $(x \rd y)' = y \rd x$. In fact,
\begin{alignat*}{2}
(x\rd y)' &\by{reconstruct} (\underbrace{x\rd y})\ld ((x\rd y)\rd (x\rd y))
&&\by{step65} ((x\rd y)\rd (y\rd y))\ld ((x\rd y)\rd (x\rd y)) \\
&\byx{(T1)} (y\rd y)\rd (\underbrace{((x\rd y)\rd (x\rd y))\ld (x\rd y)})
&&\byx{(T2)} (y\rd y)\rd (x\rd y) \\
&\by{t4-goal4} (y\rd y)'\ld (x\rd y)'
&&\by{step43} (y\rd y)'\ld (y\rd x) \\
&\by{t4-goal2} (y\rd y)\ld (y\rd x)
&&\by{step36} y\rd x\,.
\end{alignat*}
Thus we have established (T4), completing the proof.
\end{proof}

\section{Open Problems}

We begin by restating a couple of problems from \cite{AM}. A set of identities in two binary operations is
said to be \emph{semi-separated} if at most two of the identities involve both operations.

\begin{problem}
Is there a semi-separated set of identities characterizing inverse semigroups in terms of
their division bimagmas?
\end{problem}

Of course, one can also ask this sort of question about the other varieties of unary semigroups that we considered in this paper.

$ \ $

A $3$-basis characterizing inverse semigroups in terms of their
division bimagmas was presented in \cite{AM}.

\begin{proposition}
\prplabel{tamura2}
Inverse semigroups are characterized in terms of their division bimagmas
by the independent identities \emph{(T1)}, \emph{(T2)} and
\[
(x \rd x) \ld (y \ld y) = y \ld (y \rd (x \rd x))\,. \tag{T6}
\]
\end{proposition}

The proof, found by \textsc{Prover9}, was left to the
companion website of \cite{AM}. That proof shows that the
axiom set $\{$(T1), (T2), (T6)$\}$ is equivalent to $\{$(T1),\ldots,(T5)$\}$.
We tried to find a shorter automated proof that would be easier to translate
into humanly readable form, but were unable to find anything reasonable.
As current research in automated deduction aims to find tools that provide mathematical insight into theorems and their proofs, we feel that this problem should be posed as a test question for those researchers.

\begin{problem}
Find a humanly understandable proof of Proposition \prpref{tamura2}.
Alternatively, find another $3$-basis for the division bimagmas of
inverse semigroups with a humanly understandable proof.
\end{problem}

As noted above, one of the goals of defining a class of algebras in terms of a class of algebras of a different type is that many properties/problems become obvious in the new setting, while difficult to spot in the first. 
For instance, Tamura proved that for the division bimagmas of regular involuted semigroups  the following are equivalent:
\begin{enumerate}
 \item $(S,\ld)$ is associative;
\item $(S,\rd)$ is associative;
\item $(S,\rd)=(S,\ld)$;
\end{enumerate}

In addition the class of regular involuted semigroups in which  $(S,\rd)=(S,\ld)$ is the class of commutative semigroups satisfying $x^3=x$. Therefore the following problems are natural.

\begin{problem}
For each class of semigroups defined in this paper, characterize when:
\begin{enumerate}
 \item $(S,\ld)$ [$(S,\rd)$] is associative [commutative, idempotent, with identity, with zero, nilpotent $(x \ld (x \ld\ldots (x \ld x))\ldots ))=0$ or its mirror image, E-unitary];
 \item $(S,\ld)$ and $(S,\rd)$ are equal [dual ($x\rd y = y \ld x$), isomorphic, anti-isomorphic];
\item $(S,\ld)$ distributes over $(S,\rd)$ (that is, for example, $(x \rd y) \ld z = (x \ld z ) \rd (y \ld z)$).
\end{enumerate}
\end{problem}

\begin{problem}
For the classes of semigroups discussed in this paper, characterize the natural partial order on a semigroup in the language of bimagmas. Similarly characterize Green's relations.
\end{problem}

\begin{acknowledgment}
We are pleased to acknowledge the assistance of the automated deduction tool
\textsc{Prover9} and the finite model builder \textsc{Mace4}, both developed by
McCune \cite{Prover9}.

The first author was partially supported by FCT and FEDER, Project POCTI-ISFL-1-143 of Centro de Algebra da Universidade de Lisboa, and by FCT and PIDDAC through the project PTDC/MAT/69514/2006.
\end{acknowledgment}

\end{document}